\documentclass[12pt]{article}

\usepackage{amsmath,amsthm,amssymb,fullpage}

\newtheorem{thm}{Theorem}[section]
\newtheorem{lem}[thm]{Lemma}
\newtheorem{conj}[thm]{Conjecture}
\newtheorem{prop}[thm]{Proposition}

\newtheorem{example}[thm]{Example}
\newtheorem{property}[thm]{Property}

\numberwithin{equation}{section}

\begin{document}

\title{On a Conjecture of Brouwer Involving the Connectivity of Strongly Regular Graphs}
\author{Sebastian M. Cioab\u{a}\footnote{Department of Mathematical Sciences, University of Delaware, Newark, DE 19707-2553, USA},\,
Kijung Kim\footnote{Department of Mathematics, Pusan National University, Busan 609-735, South Korea}
\,and Jack H. Koolen\footnote{Department of Mathematics, POSTECH, Pohang 790-785, South Korea}}
\footnotetext{E-mail addresses: {\tt cioaba@math.udel.edu} (S.M. Cioab\u{a}), {\tt knukkj@pusan.ac.kr} (K. Kim) and {\tt koolen@postech.ac.kr} (J.H. Koolen)}

\date{December 28, 2011}
\maketitle

\centerline{\textit{Dedicated to the 60th birthday of Andries E. Brouwer}}

\begin{abstract}
In this paper, we study a conjecture of Andries E. Brouwer from 1996
regarding the minimum number of vertices of a strongly regular graph whose removal disconnects the graph into non-singleton components.

We show that strongly regular graphs constructed from copolar spaces and from the more general spaces called $\Delta$-spaces are counterexamples to Brouwer's Conjecture. Using J.I. Hall's characterization of finite reduced copolar spaces, we find that the triangular graphs $T(m)$, the symplectic graphs $Sp(2r,q)$ over the field $\mathbb{F}_q$ (for any $q$ prime power), and the strongly regular graphs constructed from the hyperbolic quadrics $O^{+}(2r,2)$ and from the elliptic quadrics $O^{-}(2r,2)$ over the field $\mathbb{F}_2$, respectively, are counterexamples to Brouwer's Conjecture. For each of these graphs, we determine precisely the minimum number of vertices whose removal disconnects the graph into non-singleton components. While we are not aware of an analogue of Hall's characterization theorem for $\Delta$-spaces, we show that complements of the point graphs of certain finite generalized quadrangles are point graphs of $\Delta$-spaces and thus, yield other counterexamples to Brouwer's Conjecture.

We prove that Brouwer's Conjecture is true for many families of strongly regular graphs including the conference graphs, the generalized quadrangles $GQ(q,q)$ graphs, the lattice graphs, the Latin square graphs, the strongly regular graphs with smallest eigenvalue $-2$ (except the triangular graphs) and the primitive strongly regular graphs with at most 30 vertices except for few cases.

We leave as an open problem determining the best general lower bound for the minimum size of a disconnecting set of vertices of a strongly regular graph, whose removal disconnects the graph into non-singleton components.
\end{abstract}

\section{Introduction}\label{sec:intro}

Strongly regular graphs are interesting mathematical objects with numerous connections to combinatorics, algebra, geometry, coding theory and computer science among others (see \cite{BCN,BH1, BL,Cam,GR,S1,VW}). According to Cameron \cite{Cam} (see also \cite{Cam2}), {\em strongly regular graphs form an important class of graphs which lie somewhere between the highly structured and the apparently random}.

A graph $G$ is a strongly regular graph with parameters $v, k, \lambda$ and $\mu$ (shorthanded $(v,k,\lambda,\mu)$-SRG for the rest of the paper) if it has $v$ vertices, is $k$-regular, any two adjacent vertices have exactly $\lambda$ common neighbors and any two non-adjacent vertices have exactly $\mu$ common neighbors.

A set of vertices $S$ of a connected, non-complete graph $G$ is called a disconnecting set (also known as a vertex separator, separating set, vertex cut or vertex cutset in the literature) if removing the vertices of $S$ and the edges incident with them will make the resulting graph disconnected. The vertex-connectivity of a connected and non-complete graph $G$ equals the minimum size of a disconnecting set of $G$. This is a well studied combinatorial invariant which is related to important algebraic parameters of $G$ such as its eigenvalues. The connection between eigenvalues and vertex-connectivity is one of the classical results of spectral graph theory which originated with Fiedler \cite{F} and has been investigated in various contexts by many researchers (see, for example, Alon \cite{Al}, Haemers \cite{H1}, Helmberg, Mohar, Poljak and Rendl \cite{HMPR}, Krivelevich and Sudakov \cite{KS} or Tanner \cite{Tan}).

Brouwer and Mesner \cite{BM} used Seidel's \cite{S} classification of strongly regular graphs with eigenvalues at least $-2$ to prove that the vertex-connectivity of any connected strongly regular graph of degree $k$ equals its degree $k$. Moreover, Brouwer and Mesner showed that the only disconnecting sets of size $k$ are the sets of all neighbors of a given vertex of the graph. Their work was extended recently by Brouwer and Koolen \cite{BK} who proved the same result for distance-regular graphs thus, settling an open problem of Brouwer from \cite{Br1}. This work is a contribution towards solving two important open problems in the area which are the conjectures of Godsil and respectively Brouwer stating that the edge-connectivity (respectively the vertex-connectivity) of any connected class of an association scheme equals its degree (see \cite{Br1} for more details).

In view of these results, a natural problem is to determine the minimum size of a disconnecting set $S$ a connected $(v,k,\lambda,\mu)$-SRG that does not contain the neighborhood of any vertex $x\notin S$. This is equivalent to finding the minimum size  of a disconnecting set
whose removal disconnects the graph into non-singleton components. We denote by $\kappa_2(G)$ the minimum size of such a disconnecting set of a connected graph $G$ if such sets exists. Note that for some graphs $G$ (such as complete bipartite graphs or some strongly regular graphs from Section \ref{sec:general}), such disconnecting sets do not exist. This parameter has been investigated for other interesting classes of graphs such as minimal Cayley graphs (see the work of Hamidoune, Llad\'{o} and Serra \cite{HLS} for example). In the case of a connected $(v,k,\lambda,\mu)$-SRG, a natural candidate for the value of $\kappa_2(G)$ would be $2k-\lambda-2$ as this equals the size of the neighborhood of an edge of the graph. This was actually formulated as a conjecture in 1996 by Andries Brouwer \cite{Br1}.
\begin{conj}[Brouwer \cite{Br1}]
Let $G$ be a connected $(v,k,\lambda,\mu)$-SRG, and let $S$ be a
disconnecting set of $G$ whose removal disconnects $G$ into non-singleton components. Show that $|S|\geq 2k-\lambda-2$.
\end{conj}

In this paper, we use algebraic, combinatorial and geometric methods to study Brouwer's Conjecture. We show the conjecture is false in general by proving that strongly regular graphs constructed from copolar spaces and from $\Delta$-spaces (see Section \ref{copolarsec} for details) are counterexamples. Using J.I.Hall's characterization of finite reduced copolar spaces (see \cite{Hall1} and Section \ref{copolarsec}), we present four infinite families of counterexamples to Brouwer's Conjecture: the triangular graphs $T(m)$ (see Section~\ref{sec:triangular}), the symplectic graphs $Sp(2r,q)$ over $\mathbb{F}_q$ (see Section~\ref{sec:symplectic} for their definition and more details), the strongly regular graphs obtained from the hyperbolic quadrics $O^{+}(2r,2)$ over $\mathbb{F}_2$ (see Section \ref{sec:hyperbolic}) and the strongly regular graphs obtained from the elliptic quadrics $O^{-}(2r,2)$ over $\mathbb{F}_2$ (see Section~\ref{sec:elliptic}). For each graph above, we determine precisely the minimum number of vertices whose removal disconnects the graph into non-singleton components. We also discuss counterexamples coming from $\Delta$-spaces and show that the complements of the point graphs of certain generalized quadrangles yield other counterexamples to Brouwer's Conjecture.

It is well known that a strongly regular graph is either a conference graph (which is a $(4t+1,2t,t-1,t)$-SRG) or all its eigenvalues are integers (see \cite[Section 10.3]{GR}).  Results of Bose and Neumaier (see Bose \cite{Bose}, Neumaier \cite{Neumaier} or \cite{BH1}) imply that for a fixed negative integer $-m$, there are finitely many strongly regular graphs with smallest eigenvalue $-m$ that are not obtained from an orthogonal array $OA(t,n)$ or as a block graph of a Steiner system.

Motivated by these facts, we show that Brouwer's Conjecture is true for many interesting strongly regular graphs including the conference graphs, the generalized quadrangles $GQ(q,q)$ graphs, the $OA(2,n)$ lattice graphs, the $OA(3,n)$ Latin square graphs, the strongly regular graphs with smallest eigenvalue $-2$ (except the triangular graphs) and the primitive strongly regular graphs with at most 30 vertices except for few cases. We plan to investigate the status of Brouwer's Conjecture for general $OA(t,n)$ strongly regular graphs (with $t\geq 4$) and for block graphs of Steiner systems in a future work.

Our paper is organized as follows. In Section~\ref{sec:general}, we give some sufficient conditions stated only in terms of $v, k,\lambda$ and $\mu$ under which Brouwer's Conjecture is true. As a consequence of our results in Section~\ref{sec:general} we show that many families of strongly regular graphs including the conference graphs (and consequently Paley graphs), the strongly regular graphs obtained from generalized quadrangles $GQ(q,q)$ or the complements of the symplectic graphs over $\mathbb{F}_q$ satisfy Brouwer's Conjecture. In Section~\ref{copolarsec}, we give the definition of copolar spaces and show that strongly regular graphs constructed from such spaces are counterexamples to Brouwer's Conjecture. J.I. Hall \cite{Hall1} has classified all these strongly regular graphs and these counterexamples are described in detail in Sections \ref{sec:triangular}, \ref{sec:symplectic}, \ref{sec:hyperbolic} and \ref{sec:elliptic}. For each such counterexample $G$, we compute the exact value of $\kappa_2(G)$. In Section~\ref{copolarsec}, we also describe $\Delta$-spaces which are a generalization of copolar spaces and show that strongly regular graphs obtained from such spaces are counterexamples to Brouwer's Conjecture. We are not aware of a classification of the strongly regular graphs arising from finite $\Delta$-spaces, but we can present some examples of such strongly regular graphs, namely the complements of the point-graph of certain finite generalized quadrangles, which are also counterexamples to Brouwer's Conjecture. In Section~\ref{sec:triangular}, we describe the triangular graphs $T(m)$ and determine $\kappa_2(T(m))$. In Section~\ref{sec:symplectic}, we describe the symplectic graphs $Sp(2r,q)$ over $\mathbb{F}_q$ and compute $\kappa_2(Sp(2r,q))$. In Section~\ref{sec:hyperbolic}, we describe the strongly regular graphs obtained from the hyperbolic quadric $O^{+}(2r,2)$ over $\mathbb{F}_2$ and determine $\kappa_2(O^{+}(2r,2))$. In Section~\ref{sec:elliptic}, we describe the strongly regular graphs obtained from the elliptic quadric $O^{-}(2r,2)$ over $\mathbb{F}_2$ and calculate $\kappa_2(O^{-}(2r,2))$. In Section~\ref{sec:lattice} and \ref{sec:oa}, we show that Brouwer's Conjecture is true for the $OA(2,n)$ lattice graphs and the $OA(3,n)$ Latin square graphs, respectively. In Section~\ref{sec:primitive}, we determine the status of Brouwer's Conjecture for the primitive strongly regular graphs with at most 30 vertices. We conclude our paper with some final remarks and open questions in Section~\ref{sec:final}.

\section{Disconnecting sets in strongly regular graphs}\label{sec:general}

Our graph theoretic notation is standard (for undefined notions see \cite{BH1,GR}). The adjacency matrix of a graph $G$ has its rows and columns indexed after the vertices of the graph and its $(u,v)$-th entry equals $1$ if $u$ and $v$ are adjacent and $0$ otherwise. If $G$ is a connected $k$-regular graph of order $v$, it is known (see \cite{BCN,BH1,GR}) that $k$ is the largest eigenvalue of the adjacency matrix of $G$ and its multiplicity is $1$. In this case, let $k=\theta_1>\theta_2\geq \dots \geq \theta_v$ denote the eigenvalues of the adjacency matrix of $G$.
If $G$ is a $(v,k,\lambda,\mu)$-SRG, then it is known that $G$ has exactly three distinct eigenvalues; let $k>\theta_2>\theta_v$ be the distinct eigenvalues of $G$, where $\theta_2=\frac{1}{2}\left(\lambda-\mu+\sqrt{(\lambda-\mu)^2+4(k-\mu)}\right)$ and $\theta_v=\frac{1}{2}\left(\lambda-\mu-\sqrt{(\lambda-\mu)^2+4(k-\mu)}\right)$ (see \cite{BCN,BH1,GR} for details). Thus, $\theta_2+\theta_v=\lambda-\mu$ and $\theta_2\theta_v=\mu-k$ which imply $\lambda=k+\theta_2+\theta_v+\theta_2\theta_v$ and $\mu=k+\theta_2\theta_v$.

If $G$ is a graph with vertex set $V(G)$ and $X\subset V(G)$, let $N(X)=\{y\notin X: y\sim x \text{ for some } x \in X\}$ denote the neighborhood of $X$. If $G$ is a $(v,k,\lambda,\mu)$-SRG, then $|N(\{u,v\})|=2k-\lambda-2$ for every edge $uv$ of $G$. Recall that $\kappa_2(G)$ denotes the minimum size of a disconnecting set of $G$ whose removal disconnects the graph into non-singleton components.

Let $G$ be a $(v,k,\lambda,\mu)$-SRG. We say that $G$ is OK if either it has no disconnecting set such that each component has as at least two vertices, or if $\kappa_2(G) = 2k - \lambda - 2$. If $G$ is a $(v,k,\lambda,\mu)$-SRG, then its complement $\overline{G}$ is a $(\overline{v},\overline{k},\overline{\lambda},\overline{\mu})$-SRG, with the following parameters $\overline{v} = v, \overline{k} = v-k-1, \overline{\lambda} = v - 2k
+ \mu -2, \overline{\mu} = v - 2k + \lambda$ (see \cite[p.9]{BCN} or \cite[p.218]{GR}). Our next result shows that graphs with a small number of vertices are OK.

\begin{lem}\label{smallv}
If $G$ is a $(v,k,\lambda,\mu)$-SRG with $v \leq 2k - \lambda + 2$, then $G$ is OK.
\end{lem}
\begin{proof}
If $G$ has a disconnecting set $S$ such that each component of $G\setminus S$ has at
least three vertices, then let $x$ and $y$ be two adjacent vertices in one of such component.
Then $x$ and $y$ are at distance two in $\overline{G}$ and they have at least $3$ common
vertices in $\overline{G}$ (as they are adjacent to all the other vertices in the
other components of $G\setminus S$). This implies  $\overline{\mu} \geq 3$ and hence $v\geq 2k - \lambda +3$, a contradiction.
\end{proof}

We outline here the methods we will use throughout the paper.

Let $G$ be a connected graph. If $S$ is a disconnecting set of $G$ of minimum size such that the components of $G\setminus S$ are not singletons, then let $A$ denote the vertex set of one of the components of $G\setminus S$ of minimum size. By our choice of $A$, we have that $|B|\geq |A|$, where $B:=V(G)\setminus (A\cup S)$. As $S$ is a disconnecting set, it follows that $N(A)\subset S$ and consequently, $|N(A)|\leq |S|$. As also pointed out to us by the referee, note that it is possible for the disconnecting set $S$ to contain a vertex $y$ and its neighborhood $N(y)$ in which case $y\in S$, but $y\notin N(A)$ and thus, $N(A)\neq S$ (see also the last section of the paper for a discussion of such disconnecting sets). 

In order to prove Brouwer's Conjecture is true for a $(v,k,\lambda,\mu)$-SRG $G$ with vertex set $V$ and $v\geq 2k-\lambda+3$, we will show that $|S|\geq 2k-\lambda-2$ for any subset of vertices $A$ with $3\leq |A|\leq \frac{v}{2}$ having the property that $A$ induces a connected subgraph of $G$. In some situations, we will be able to prove the stronger statement that $|N(A)|\geq 2k-\lambda-2$.

In order to show that Brouwer's Conjecture is false for some $(v,k,\lambda,\mu)$-SRG $H$, we will describe a subset of vertices $C$ inducing a connected subgraph of $H$ such that $3\leq |C|\leq v-3$, $|N(C)|\leq 2k-\lambda-3$ and the components of the graph obtained by removing $N(C)$ from $H$ are not singletons.

Throughout the paper, $S$ will be a disconnecting set of $G$, $A$ will stand for a subset of vertices of $G$ that induces a connected subgraph of $G\setminus S$ of smallest order and $B:=V(G)\setminus (A\cup S)$. As before, $N(A)\subset S$ and thus, $|S|\geq |N(A)|$. Let $|A|=a, |B|=b$ and $|S|=s$.

We will use the following result which relates the size of a disconnecting set to the eigenvalues of the graph.
\begin{lem}[Haemers \cite{H1}; Helmberg, Mohar, Poljak and Rendl \cite{HMPR}]\label{haemersbnd}
If $G$ is a connected $k$-regular graph, then
\begin{equation}
|S|\geq \frac{ab}{v}\cdot \frac{4(k-\theta_2)(k-\theta_v)}{(\theta_2-\theta_v)^2}
\end{equation}
\end{lem}
\noindent When applied to a strongly regular graph, the previous result yields the following:
\begin{lem}\label{HL1}
If $G$ is a connected $(v,k,\lambda,\mu)$-SRG, then
\begin{equation}
|S|\geq \frac{ab}{v}\cdot  \frac{4[k^2-(\lambda-\mu)k+\mu-k]}{(\lambda-\mu)^2+4(k-\mu)}=\frac{4ab\mu}{(\lambda-\mu)^2+4(k-\mu)}.
\end{equation}
\end{lem}
\begin{proof}
Because $G$ is a $(v,k,\lambda,\mu)$-SRG, we know that $\theta_2+\theta_v=\lambda-\mu$, $\theta_2\theta_v=\mu-k$ and $v=1+k+k(k-\lambda-1)/\mu$. The conclusion now follows from Lemma \ref{haemersbnd}.
\end{proof}

Lemma \ref{HL1} can be used to show that Brouwer's Conjecture is true when the following condition is satisfied:
\begin{prop}\label{kSRG}
Let $G$ be a connected $(v,k,\lambda,\mu)$-SRG. If
\begin{equation}\label{kSRGeq}
4(k-2\lambda)(k-\mu)>(\lambda-\mu)^2(2k-\lambda-3)
\end{equation}
then $G$ is OK.
\end{prop}
\begin{proof}
Let $s$ denote the minimum size of a disconnecting set $S$ whose removal leaves only non-singleton components. Assume that $s\leq 2k-\lambda-3$. This implies $a+b=v-s\geq v-(2k-\lambda-3)=v+3+\lambda-2k$. As $a,b\geq 3$, we obtain $ab\geq 3(v+\lambda-2k)$. This inequality and Lemma \ref{haemersbnd} imply
\begin{equation}\label{haemers1}
s\geq \frac{4ab\mu}{(\lambda-\mu)^2+4(k-\mu)}\geq \frac{12\mu(v+\lambda-2k)}{(\lambda-\mu)^2+4(k-\mu)}.
\end{equation}
Since $v=1+k+k(k-\lambda-1)/\mu$, this gives
\begin{align*}
s&\geq 
\frac{12k^2+(-12\mu-12\lambda-12)k+12\mu+12\lambda\mu}{(\lambda-\mu)^2+4(k-\mu)}>2k-\lambda-3.
\end{align*}
where the last inequality can be shown to be equivalent to our hypothesis \eqref{kSRGeq}. Thus, $s>2k-\lambda-3$ which is a contradiction. This finishes our proof.
\end{proof}

Proposition \ref{kSRG} shows that Brouwer's Conjecture is true for all known triangle-free strongly regular graphs: the Petersen graph $(10,3,0,1)$-SRG, the folded $5$-cube $(16,5,0,2)$-SRG, the Hoffman-Singleton graph $(50,7,0,1)$-SRG, the Gewirtz graph $(56,10,0,2)$-SRG, the $M_{22}$ graph $(77,16,0,4)$-SRG, and the Higman-Sims graph $(100,22,0,6)$-SRG (see \cite{BH1} for a detailed description of these graphs).

Proposition \ref{kSRG} implies that Brouwer's Conjecture is also true for infinite families of strongly regular graphs such as the complements of symplectic graphs over $\mathbb{F}_q$ (these graphs are $\left(\frac{q^{2r}-1}{q-1},\frac{q^{2r-1}-q}{q-1},\frac{q^{2r-2}-1}{q-1}-2,\frac{q^{2r-2}-1}{q-1}\right)$-SRGs and their complements are described in Section \ref{sec:symplectic}) and the generalized quadrangle $GQ(q,q)$ graphs which are $((q+1)(q^2+1),q(q+1),q-1,q+1)$-SRGs for $q$ a prime power (see \cite{BH1,GR} for a more detailed description of these graphs).

Proposition \ref{kSRG} also implies that Brouwer's Conjecture is true for many strongly regular graphs where $k\geq 2\lambda+1$ and $\lambda$ and $\mu$ are close to each other.

\begin{lem}\label{diff}
Let $G$ be a $(v,k,\lambda,\mu)$-SRG. If $\lambda-\mu\in \{-1,0,1\}$ and $k\geq 2\lambda+1$, then $\kappa_2(G)=2k-\lambda-2$.
\end{lem}

\begin{proof}
If $\lambda=\mu$, then by Proposition \ref{kSRG}, we have $\kappa_2(G)=2k-\lambda-2$ whenever $4k^2-(8\lambda+4\mu)k+8\lambda\mu\geq 1$. This is equivalent with $4(k-2\lambda)(k-\mu)\geq 1$ which is certainly true as $k\geq 2\lambda+1$ and $k>\mu$.

If $\lambda=\mu+1$, then by Proposition \ref{kSRG}, we have $\kappa_2(G)=2k-\lambda-2$ whenever $4(k-2\lambda)(k-\mu)\geq 1+2k-\lambda-3=2k-\lambda-2$. This is true when $k\geq 2\lambda+1$ as
$$
4(k-2\lambda)(k-\lambda+1)\geq 4(k-\lambda+1)\geq 2k-\lambda-2
$$
where the last inequality is equivalent to $2k\geq 3\lambda-6$ (true as $k\geq 2\lambda+1$).

If $\lambda=\mu-1$, then by Proposition \ref{kSRG}, we have $\kappa_2(G)=2k-\lambda-2$ whenever $4(k-2\lambda)(k-\mu)\geq 1+2k-\lambda-3=2k-\lambda-2$. This is true when $k\geq 2\lambda+1$ as
$$
4(k-2\lambda)(k-\lambda-1)\geq 4(k-\lambda-1)\geq 2k-\lambda-2
$$
where the last inequality is equivalent to $2k\geq 3\lambda +2$ (true as $k\geq 2\lambda+1$).
\end{proof}

The previous result implies that Brouwer's Conjecture is true for conference graphs (which are $(4t+1,2t,t-1,t)$-SRGs). This family includes the Paley graphs among others (see \cite{BH1,GR} for a description of these graphs).

\section{Copolar and $\Delta$-spaces}\label{copolarsec}

A pair $(P,L)$, where $L \subseteq 2^P$, is called a partial linear space if (i) every $\ell \in L$ contains at least two points in $P$, and (ii) for two distinct $p, q \in P$ there is at most one line $\ell \in L$ that contains both. We call the elements of $P$ points and the elements in $L$ lines. A point $p$ is on the line $\ell$ if $p \in \ell$. Also, two distinct points are collinear if there is a line that contains both points. A partial linear space $(P,L)$ is called a {\em copolar space} (following Hall \cite{Hall1}) or {\em proper delta space} (according to Higman; see Hall \cite{Hall1} and the references therein) if for any point $p$ and line $\ell$, $p\notin \ell$, $p$ is collinear with none or all but one of the points of $\ell$. A more general notion is the notion of a $\Delta$-space. A partial linear space $(P, L)$ is called a {\em $\Delta$-space}  if
for any point $p$ and line $\ell$, $p \not \in \ell$, $p$ is collinear with none, all but one or all the points of $\ell$. We say a partial linear space $(P, L)$ is of order $(s,t)$ if every line contains exactly $s+1$ points, and every point is in exactly $t+1$ lines.

Assume that the point graph $\Gamma$ of a $\Delta$-space of order $(s,t)$ (i.e. the graph with vertex point $P$ where two points are adjacent if they are collinear) is strongly regular with parameters $(v, k, \lambda, \mu)$ with $k=s(t+1)$. A line $\ell$ (which is a clique of order $s+1$ in $\Gamma$) has exactly $2k -\lambda - s-1$ neighbors (vertices adjacent to at least one vertex of the clique corresponding to $\ell$). Let $x$ be a vertex not collinear with any point on $\ell$. Then there are exactly $\mu(s+1)$ paths of length two with $x$ as one endpoint and a point of $\ell$ as its other endpoint. This means that $x$ has at most $\mu(s+1)/s$ neighbors at distance one from $\ell$. So, if $\mu(s+1)/s < k$ and $s\geq 2$, then $\Gamma$ is a counterexample for the conjecture of Brouwer. Thus, any $(v,k,\lambda,\mu)$-SRG that is the point graph of a $\Delta$-space of order $(s,t)$ and satisfies the conditions $\mu(s+1)/s<k$ and $s\geq 2$, is a counterexample to Brouwer's Conjecture. We briefly describe such graphs below and in more detail in the later sections.

J.I. Hall \cite{Hall1} determined all the strongly regular graphs that appear as the point graph of a copolar space and these graphs are: the triangular graphs $T(m)$ (see Section \ref{sec:triangular}), the symplectic graphs $Sp(2r,q)$ over the field $\mathbb{F}_q$ for any $q$ prime power (see Section \ref{sec:symplectic}), the strongly regular graphs constructed from the hyperbolic quadrics $O^{+}(2r,2)$ (see Section \ref{sec:hyperbolic}) and from the elliptic quadrics $O^{-}(2r,2)$ over the field $\mathbb{F}_2$ (see Section \ref{sec:elliptic}) respectively, and the complements of Moore graphs. Only the complement of a Moore graph does not satisfy $\mu(s+1)/s<k$, so all the other graphs give all counterexamples for Brouwer's Conjecture. In the next four sections, we will describe these counterexamples in more detail and we will compute the exact value of $\kappa_2$ for all of them.

Below we will give some examples of $\Delta$-spaces coming from the hyperbolic lines of a generalized quadrangle. A generalized quadrangle of order $(s,t)$, $GQ(s,t)$, is a partial linear space $(P, L)$ of order $(s,t)$ such that  for any point $p$ and line $\ell$, $p \not \in \ell$, $p$ is collinear with exactly one point on $\ell$. We will also call the corresponding point graph $\Gamma$ of a generalized quadrangle of order $(s,t)$, a $GQ(s,t)$. A pair of non-adjacent vertices $(x,y)$ is called {\em regular} if the hyperbolic line $\{x,y\}^{\perp \perp}$ has size $t+1$, where $x^{\perp} = \{x\}^{\perp}= \{x\} \cup \Gamma(x)$, and $A^{\perp} = \bigcap_{a \in A} a^{\perp}$, for $x$ a vertex and $A$ a set of vertices. Note that in this case the induced subgraph on $\{x,y\}^{\perp} \cup \{x,y\}^{\perp \perp}$ is a $K_{t+1, t+1}$. A vertex is called {\em regular} if for all $y$ non-adjacent to $x$ the pair $(x,y)$ is regular. If a generalized quadrangle of order $(s,t)$, $\Gamma$, has the property that every vertex is regular then $\Pi = (V(\Gamma), \{$hyperbolic lines$\})$ is a $\Delta$-space. This means that the complement of $\Gamma$ is a counterexample for Brouwer's Conjecture as it is just the point graph of $\Pi$ (as in the complement of $\Gamma$ the hyperbolic line is just a clique with $t+1$ vertices and the subgraph induced on the vertices not adjacent to this clique is just the induced subgraph on $\{x,y\}^{\perp}$ which is another hyperbolic line or clique with $t+1$ vertices; in this case, we do not need the hyperbolic lines to have size $t+1$, it is sufficient that they have size at least three, in order to give a counterexample.).  More generally, with the same proof as above, the complement of a generalized quadrangle of order $(s,t)$ with one regular pair $(x,y)$ of points is a counterexample to Brouwer's Conjecture. There are examples of generalized quadrangles with all points regular and generalized quadrangles with only one regular point, see \cite[p.26-28]{PayneThas}. So far as the authors know, the generalized quadrangles of order $(s,t)$ with all points regular are not classified.

\section{The triangular graphs $T(m)$}\label{sec:triangular}

The triangular graph $T(m)$ is the line graph of the complete graph $K_m$; its vertices are the $2$-subsets of $[m]:=\{1,\dots,m\}$ and $\{u,v\}\sim \{x,y\}$ if and only if $|\{u,v\}\cap \{x,y\}|=1$. It is easy to see that the triangular graph $T(m)$ is a $\left({m\choose 2},2(m-2),m-2,4\right)$-SRG and it is actually known that for any $m\neq 8$, any $\left({m\choose 2},2(m-2),m-2,4\right)$-SRG is isomorphic to $T(m)$ (see \cite{BH1,VW}). When $m=8$, there are $3$ pairwise non-isomorphic strongly regular graphs known as Chang graphs whose parameters are $(28,12,6,4)$ which are not isomorphic to $T(8)$ (see \cite{BH1,VW}). In  Section \ref{sec:primitive}, we prove that the Chang graphs are OK therefore showing that $\kappa_2(G)$ is not determined by the parameters of $G$.

For $m=4$, the graph $T(4)$ is a $(6,4,2,4)$-SRG and $2k-\lambda +2=8-2+2=8$. For $m=5$, the graph $T(5)$ is a $(10,6,3,4)$-SRG and $2k-\lambda+2=12-3+2=11$. Both these graphs satisfy the condition of Lemma \ref{smallv} and thus, both $T(4)$ and $T(5)$ are OK.

Assume $m\geq 6$ from now on. Brouwer's Conjecture states that $\kappa_{2}(T(m))=2k-\lambda-2=2(2(m-2))-(m-2)-2=3(m-2)-2=3m-8$.
As shown in Section \ref{copolarsec}, this is not true for triangular graphs. In the next proposition, we determine $\kappa_2(T(m))$ precisely as well as the structure of minimum disconnecting sets.

\begin{prop}\label{triangular}
For $m\geq 6$, $\kappa_2(T(m))=3m-9$ and the only disconnecting sets of this size are formed (modulo a permutation of $[m]$) by the set of vertices adjacent to at least one of the vertices $\{1,2\},\{1,3\}$ or $\{2,3\}$.
\end{prop}
\begin{proof}
By the remarks in Section \ref{copolarsec} involving Hall's characterization of copolar spaces, we know that $T(m)$ is a counterexample to Brouwer's Conjecture and that a good candidate for $\kappa_2(T(m))$ is the size of the neighborhood of a clique of $T(m)$ corresponding to a hyperbolic line. Such a clique has the form of $C=\{\{1,2\}, \{1,3\},\{2,3\}\}$. By the remarks in Section \ref{copolarsec} or by direct observation, $|N(C)|=2k-\lambda-3=3m-9$ and $\kappa_2(T(m))\leq 3m-9$ (as $T(m)\setminus (C\cup N(C))$ is isomorphic to $T(m-3)$).

The remaining part of the proof is devoted to showing that any disconnecting set of $T(m)$ whose removal creates only non-singleton components, must have at least $3m-9$ vertices.

Recall our strategy from Section \ref{sec:general}. If $S$ is a disconnecting set of $G$ of minimum size such that the components of $G\setminus S$ are not singletons, then let $A$ denote the vertex set of one of the components of $G\setminus S$ of minimum size. We have $N(A)\subset S$ and thus, $|N(A)|\leq |S|$ and by our choice of $A$, we have that $|B|\geq |A|$, where $B:=V(G)\setminus (A\cup S)$.

We will show that $|S|\geq 3m-9$ with equality if and only if $S$ is (modulo a permutation of $[m]$) $N(A)$ where $A=\{\{1,2\}, \{2,3\}, \{1,3\}\}$. Our proof uses a case analysis depending on $A$ inducing a clique or not.

If $A$ induces a clique, then without any loss of generality, we may assume we are in one of the following two situations:
\begin{enumerate}
\item $A=\{\{1,2\},\{1,3\},\{2,3\}\}$

In this case, $|N(A)|=3m-9$ which finishes the proof.

\item $A=\{\{1,2\},\dots,\{1,a+1\}\}$ and  $3\leq a\leq m-3$.

The size of the neighborhood of $A$ is
\begin{equation*}
|N(A)|={m\choose 2}-a-{m-a-1\choose 2}=:f(a).
\end{equation*}
Since $f(a)$ attains its minimum for $a=3$, it follows that $|S|\geq |N(A)|\geq f(3)=4m-13>3m-8$ as $m\geq 6$.
\end{enumerate}

If $A$ does not induce a clique, then we may also assume that any component of $B$ is not a clique (as otherwise we can repeat the argument from the previous case). Thus, each of $A$ and $B$ must contain an induced $K_{1,2}$ (that is, a vertex adjacent to two non-adjacent vertices). Without loss of generality, we may assume that $A$ contains $\{\{1,2\},\{1,3\},\{2,4\}\}$ and $B$ contains $\{\{5,6\},\{5,7\},\{6,8\}\}$. This shows that for $m\in \{6,7\}$, $A$ or $B$ will be a clique and we are in the previous case.

Assume $m\geq 8$. The following sets of paths with one endpoint in $\{\{1,2\},\{1,3\},\{2,4\}\}$ and the other endpoint in $\{\{5,6\},\{5,7\},\{6,8\}\}$
have no interior vertices in common:
\begin{align*}
P_1&=\{\{1,2\},\{1,x\},\{5,x\},\{5,6\}: 9\leq x\leq m\}\\
P_2&=\{\{1,2\},\{2,x\},\{6,x\},\{5,6\}: 9\leq x\leq m\}\\
P_3&=\{\{1,3\},\{3,x\},\{7,x\},\{5,7\}: 9\leq x\leq m\}\\
P_4&=\{\{2,4\},\{4,x\},\{8,x\},\{6,8\}: 9\leq x\leq m\}\\
P_5&=\{\{u,v\},\{v,w\},\{w,z\}: \{u,v\}=\{1,3\} \text{ or } \{2,4\}, \{v,w\}\in \{1,2,3,4\}\times \{5,6,7,8\}, \\ 
&\qquad \{v,w\}=\{5,7\} \text{ or } \{6,8\} \}.
\end{align*}
For $1\leq i\leq 4$, each of $P_i$ contains $m-8$ paths. The set $P_5$ contains $16$ paths determined by the middle vertex. Hence, we have $4(m-8)+16=4m-16$ (interior) vertex-disjoint paths from $\{\{1,2\},\{1,3\},\{2,4\}\}$ to $\{\{5,6\},\{5,7\},\{6,8\}\}$. In order to disconnect $\{\{1,2\},\{1,3\},\{2,4\}\}$ from $\{\{5,6\},\{5,7\},\{6,8\}\}$, one must remove at least one vertex from each path in $\cup_{i=1}^{5}P_i$. This implies $|S|\geq |\cup_{i=1}^{5}P_i|=4m-16\geq 3m-8$ and finishes our proof.
\end{proof}

We remark that our proof also shows that the only disconnecting sets of size $3m-8$ in $T(m)$ whose removal leaves only non-singleton components, are of the form $N(\{u,v\})$ where $u$ and $v$ are adjacent vertices of $T(m)$ for $m\geq 6$ and (modulo a permutation of $[8]$) $\{1,2,3,4\}\times \{5,6,7,8\}$ for $T(8)$.

\section{The symplectic graphs $Sp(2r,q)$ over $\mathbb{F}_q$}\label{sec:symplectic}

Let $q$ be a prime power and $r\geq 2$ be an integer. If $x$ is a non-zero (column) vector in $\mathbb{F}_q^{2r}$, denote by $[x]$ the $1$-dimensional vector subspace of $\mathbb{F}_q^{2r}$ that is spanned by $x$ and denote by $x^t$ the row vector that is the transpose of $x$. Let $M$ be the $2r\times 2r$ block diagonal matrix whose diagonal blocks are
\begin{equation}\label{M}
\begin{bmatrix}
0  & -1\\
1 & 0
\end{bmatrix}.
\end{equation}

\noindent For example, when $r=2$, the matrix $M$ is $\begin{bmatrix}0 & -1 & 0 & 0\\1 & 0 & 0 & 0\\0 & 0 & 0 & -1\\0 & 0 & 1 & 0\end{bmatrix}$.

The symplectic graph $Sp(2r,q)$ over $\mathbb{F}_q$ is the complement of the orthogonality graph of the unique non-degenerate symplectic form over $\mathbb{F}_q^{2r}$. More precisely, its vertex set is formed by the $1$-dimensional subspaces $[x]$ of $\mathbb{F}_q^{2r}$ with $[x]\sim [y]$ if and only if $x^tMy\neq 0$. This graph is called {\em symplectic} as the function $f(x,y)=x^{t}My$ is known as a symplectic form (see \cite{GR,GR2} for more details).
Note also that some authors (such as Godsil and Royle \cite{GR2} and Shult \cite{Sh}) use the name symplectic graph for the orthogonality graph of the unique non-degenerate symplectic form (which is the complement of $Sp(2r,q)$) while others (such as Rotman and Weichsel \cite{RoWe} and Tang and Wan \cite{TaWa}) use the same notation as ours.

The symplectic graph $Sp(2r,q)$ is a $\left(\frac{q^{2r}-1}{q-1},q^{2r-1},q^{2r-2}(q-1),q^{2r-2}(q-1)\right)$-SRG. We give a short proof of this fact. It is obvious that $v=\frac{q^{2r}-1}{q-1}$. For $x\in \mathbb{F}_q^{2r}$, let $[x]^{\perp}=\{y\in \mathbb{F}_q^{2r}: x^t My=0\}$. As $\dim([x]^{\perp})=2r-1$, we have that $|[x]^{\perp}|=q^{2r-1}$ for any non-zero $x\in \mathbb{F}_q^{2r}$.
To determine $k$, consider a non-zero $x\in \mathbb{F}_q^{2r}$ and note that the number of vertices adjacent to $[x]$ in $Sp(2r,q)$ equals
$(q^{2r}-|[x]^{\perp}|)/(q-1)=q^{2r-1}$. To determine $\lambda$ and $\mu$, let $[x]\neq [y]$ be two distinct vertices of $Sp(2r,q)$. By inclusion-exclusion, $|[x]^{\perp}\cup [y]^{\perp}|=|[x]^{\perp}|+|[y]^{\perp}|-|[x]^{\perp}\cap [y]^{\perp}|=2q^{2r-1}-q^{2r-2}$ and consequently, the number of common neighbors of $[x]$ and $[y]$ equals $(q^{2r}-|[x]^{\perp}\cup [y]^{\perp}|)/(q-1)=q^{2r-2}(q-1)$.

Note that for the graph $Sp(2r,q)$, the size of the neighborhood of an edge is $2k-\lambda-2=q^{2r-1}+q^{2r-2}-2$. In Proposition \ref{sympl}, we prove that $\kappa_2(Sp(2r,q))=q^{2r-1}+q^{2r-2}-q-1$ and characterize all disconnecting sets of this size. This will show the existence of strongly regular graphs $G$ for which the difference between the conjectured value of $\kappa_2(G)$ and its actual value is arbitrarily large. In the proof of Proposition \ref{sympl} which is the main result of this section, we will use the following simple and general combinatorial result. 
\begin{lem}\label{lowerbounddiscset}
If $G$ is a $(v,k,\lambda,\mu)$-SRG and $A\subset V(G)$ induces a connected subgraph such that $|A|=t\geq 3$, then $|N(A)|\geq 2k-\lambda-t$.

If equality occurs and $A$ contains two non-adjacent vertices, then $k\leq \lambda+\mu$. If equality occurs and $A$ induces a clique in $G$ and one of the components of $G\setminus (S\cup A)$ is a singleton, then $\mu/k\geq 1-1/t\geq 2/3$.
\end{lem}
\begin{proof}
As $A$ induces a connected subgraph of $G$, consider an edge whose endpoints $x$ and $y$ are in $A$. Then $|N(A)|\geq |N(\{x,y\})|-|A\setminus \{x,y\}|=2k-\lambda-2-(t-2)=2k-\lambda-t$.

When equality happens $|N(A)|=2k-\lambda-t$ and $A$ induces a connected subgraph of $G$ that is not a clique, consider three vertices $x, y, z$ of $A$ such that $y$ is adjacent to both $x$ and $z$ while $x$ and $z$ are not adjacent. As $A\setminus \{x,y\}\subset N(\{x,y\})$ and $A\setminus \{y,z\}\subset N(\{y,z\})$, we deduce that any vertex of $(N(A)\cup A)\setminus \{y\}$ that is not adjacent to $y$, must belong to $(N(x)\cap N(z))\setminus \{y\}$. As the number of non-neighbors of $y$ in $(A\cup N(A))\setminus \{y\}$ is $k-\lambda-1$ and $|N(x)\cap N(z)|=\mu$, this implies $k-\lambda\leq \mu$ which was our goal.

When equality happens $|N(A)|=2k-\lambda-t$ and $A$ induces a clique of size $t$, it follows that  each vertex of $N(A)$ is adjacent to at least $t-1$ vertices of $A$. Otherwise, there exists a vertex $z\in S$ and two distinct vertices $x,y\in A$ such that $z$ is not adjacent to neither $x$ nor $y$. Because $|N(A)|=2k-\lambda-t$, we deduce that $N(\{x,y\})$ is the disjoint union of $A\setminus \{x,y\}$ and $N(A)$. This implies that $N(A)$ must be a subset of $N(\{x,y\})$. Thus, $z\in N(A)\subset N(\{x,y\})$ which is a contradiction.

Now if some component of $G\setminus (A\cup N(A))$ is a singleton $\{w\}$, then $w$ is not adjacent to any vertex of $A$ and its neighborhood must be contained in $N(A)$. Thus, any vertex of $A$ has exactly $\mu$ common neighbors with $w$ and all these common neighbors must be contained in $N(A)$. As $|A|=t$ and every vertex of $N(A)$ is adjacent with at least $t-1$ vertices in $A$, it follows by a simple counting argument that $t\mu\geq (t-1)k$ which finishes our proof.
\end{proof}

The following is the main result of this section.
\begin{prop}\label{sympl}
If $q$ is a prime power and $r\geq 2$, then $\kappa_2(Sp(2r,q))=q^{2r-1}+q^{2r-2}-q-1$ and the only disconnecting sets of this size are the neighborhoods of hyperbolic lines (which are sets of the form $N(\{[u],[v],[u+x_1v],\dots,[u+x_{q-1}v]\})$, where $[u]$ and $[v]$ are adjacent vertices and $\mathbb{F}_q\setminus \{0\}=\{x_1,\dots,x_{q-1}\}$).
\end{prop}
\begin{proof}
By the remarks in Section \ref{copolarsec} involving Hall's characterization of copolar spaces, we know that $Sp(2r,q)$ is a counterexample to Brouwer's Conjecture and that a good candidate for $\kappa_2(Sp(2r,q))$ is the size of the neighborhood of a clique of $Sp(2r,q)$ corresponding to a hyperbolic line. Such a clique has the form of $C=\{[u],[v],[u+x_1v],\dots,[u+x_{q-1}v]\}$ where $[u]$ and $[v]$ are adjacent vertices. By the remarks in Section \ref{copolarsec}, it follows that $|N(C)|=2k-\lambda-(q+1)=q^{2r-1}+q^{2r-2}-q-1$.

The fact that $Sp(2r,q)\setminus (C\cup N(C))$ has no singleton components follows from Section \ref{copolarsec}, but can be also proved using the last part of Lemma \ref{lowerbounddiscset}. Note that $C$ induces a clique of size $q+1$ in $Sp(2r,q)$ and $N(C)=(q+1)(q^{2r-2}-1)=2k-\lambda-(q+1)$ so the last part of Lemma \ref{lowerbounddiscset} can be applied here. As $\mu/k=1-1/q<1-1/(q+1)$, it follows from Lemma \ref{lowerbounddiscset}, that $G\setminus (C\cup N(C))$ has no singleton components.
This shows $\kappa_2(Sp(2r,q))\leq (q+1)(q^{2r-2}-1)$.

In the second part of the proof, we will show that $\kappa_2(Sp(2r,q))\geq q^{2r-1}+q^{2r-2}-q-1$. Moreover, we will prove that $|S|\geq q^{2r-1}+q^{2r-2}-q$ unless $S$ is the neighborhood of a hyperbolic line. 

We first prove that for any subset of vertices $A$ with $q+2\leq |A|\leq \frac{v}{2}$, we have $|S|\geq q^{2r-1}+q^{2r-2}-q$. Our proof is by contradiction and uses the eigenvalue methods from Lemma \ref{HL1}. Without any loss of generality, we assume that $|A|\leq |B|$.

If $r=2$, then assume that $|S|\leq q^{3}+q^{2}-q-1$ which implies $|A|+|B|=v-|S|\geq 2q+2$. As $|A|\geq q+2$, it follows that $|B|\leq q$ which contradicts the fact that $|A|\leq |B|$.

If $r\geq 3$, then assume that $|S|\leq q^{2r-1}+q^{2r-2}-q-1$ which implies $|A|+|B|=v-|S|\geq \frac{q^{2r-2}-1}{q-1}+q+1$. As $q+2\leq |A|\leq \frac{v}{2}$, it follows that $|A| |B|\geq \frac{(q+2)(q^{2r-2}-q)}{q-1}$. Lemma \ref{HL1} implies
$$
|S|\geq (q-1)|A||B|\geq (q+2)(q^{2r-2}-q)=q^{2r-1}+2q^{2r-2}-q^2-2q.
$$
However $q^{2r-1}+2q^{2r-2}-q^2-2q>q^{2r-1}+q^{2r-2}-q-1$ as this is equivalent to $q^{2r-2}>q^2+q-1$ which is true for $q\geq 2$ and $r\geq 3$. Thus, $|S|\geq q^{2r-1}+q^{2r-2}-q$ when $q+2\leq |A| \leq \frac{v}{2}$.

If $A$ induces a connected subgraph of $Sp(2r,q)$ and $3\leq |A|\leq q$, then by Lemma \ref{lowerbounddiscset}, we have that $|S|\geq 2k-\lambda-|A|\geq 2k-\lambda-q=q^{2r-1}+q^{2r-2}-q$.

The only remaining case is when $|A|=q+1$ and $A$ induces a connected subgraph of $Sp(2r,q)$. If $|A|=q+1$ and $A$ does not induce a clique, then we prove the stronger inequality $|S|\geq q^{2r-1}+q^{2r-2}-q+1$. To see this, let $x$ and $y$ be two non-adjacent vertices of $A$. Note that $|N(\{x,y\})|=2k-\mu$. Then $|S|\geq |N(A)|\geq |N(\{x,y\})|-|A\setminus \{x,y\}|=(2k-\mu)-(|A|-2)=q^{2r-1}+q^{2r-2}-q+1$ as claimed.

Thus, the only case remaining is when $|A|=q+1$ and $A$ induces a clique. By Lemma \ref{lowerbounddiscset}, $|S|\geq |N(A)|\geq 2k-\lambda-|A|=q^{2r-1}+q^{2r-2}-q-1$. Equality happens if and only the clique induced by $A$ is a hyperbolic line and $S=N(A)$. This finishes our proof.
\end{proof}

\section{The hyperbolic quadric graphs $O^{+}(2r,2)$}\label{sec:hyperbolic}

The hyperbolic quadric graph $O^{+}(2r,2)$ is the subgraph of $Sp(2r,2)$ induced by
$V^{+}:=\{(x_1,\dots,x_{2r})^t\in \mathbb{F}_2^{2r}: x_1x_2+x_3x_4+\dots +x_{2r-1}x_{2r}=1\}$ (the complement of a hyperbolic quadric in $\mathbb{F}_2^{2r}$). The vertex $x:=(x_1,\dots,x_{2r})^t$ is adjacent to $y:=(y_1,\dots,y_{2r})^t$ if $x^tMy=1$, where $M$ is defined in \eqref{M}. It is known (see \cite{GR2}) that $O^{+}(2r,2)$ is a
 $(2^{2r-1}-2^{r-1},2^{2r-2}-2^{r-1},2^{2r-3}-2^{r-2},2^{2r-3}-2^{r-1})$-SRG. The value of $2k-\lambda-2$ equals $3(2^{2r-3}-2^{r-2})-2$.

The following proposition is the main result of this section.
\begin{prop}
For $r\geq 3$, $\kappa_2(O^{+}(2r,2))=3(2^{2r-3}-2^{r-2})-3=2k-\lambda-3$. The only disconnecting sets of this size are the neighborhoods of hyperbolic lines (which are sets of the form $N(\{x,y,x+y\})$ where $x$ and $y$ are adjacent).
\end{prop}
\begin{proof}
By the remarks in Section \ref{copolarsec} involving Hall's characterization of copolar spaces, we know that $O^{+}(2r,2)$ is a counterexample to Brouwer's Conjecture and that a good candidate for $\kappa_2(O^{+}(2r,2))$ is the size of the neighborhood of a clique of $O^{+}(2r,2)$ corresponding to a hyperbolic line. Such a clique has the form of $C=\{x,y,x+y\}$ where $x$ and $y$ are adjacent in $O^{+}(2r,2)$. By the remarks in Section \ref{copolarsec}, it follows that $|N(C)|=2k-\lambda-3=3(2^{2r-3}-2^{r-2})-3$.

The fact that $O^{+}(2r,2)\setminus (C\cup N(C))$ has no singleton components follows from Section \ref{copolarsec}, but can be also proved using the last part of Lemma \ref{lowerbounddiscset}. Note that $C$ induces a clique of size $3$ in $O^{+}(2r,2)$ and $N(C)=2k-\lambda-3$ so the last part of Lemma \ref{lowerbounddiscset} can be applied here. 
As $\mu/k=(2^{2r-3}-2^{r-1})/(2^{2r-2}-2^{r-1})<\tfrac{2}{3}$, it follows by Lemma \ref{lowerbounddiscset} that $O^{+}(2r,2)\setminus N(\{x,y,x+y\})$ does not contain any singleton components.
This shows $\kappa_2(O^{+}(2r,2))\leq 3(2^{2r-3}-2^{r-2})-3$.

To show that $\kappa_2(O^{+}(2r,2))\geq 2k-\lambda-3$ and that all disconnecting sets of minimum size are neighborhoods of hyperbolic lines, we will prove that if $A$ induces a $K_{1,2}$ or a connected subgraph of order at least $4$, then $|S|>2k-\lambda-3$.

If $A$ induces a $K_{1,2}$, then assume $A=\{x,y,z\}$ such that $z$ is adjacent to both $x$ and $y$ and $x$ and $y$ are not adjacent. It follows easily that $x+z$ and $y+z$ are adjacent to $x,y$ and $z$. Thus, $|N(x)\cap N(y)\cap N(z)|\geq 2$. By inclusion and exclusion, we obtain
\begin{align*}
|S|&\geq |N(A)|\geq 2(k-1)+k-2-2\lambda-(\mu-1)+2\\
&=3\cdot 2^{2r-3}-2\cdot 2^{r-2}-1>3\cdot 2^{2r-3}-3\cdot 2^{r-2}-3\\
&=2k-\lambda-3.
\end{align*}


If $|A|\geq 4$, then Lemma \ref{HL1} implies
\begin{equation}\label{abh}
|S|\geq \frac{4ab\mu}{(\lambda-\mu)^2+4(k-\mu)}=\frac{ab(2^{r-2}-1)}{2^{r-2}+2^{r-5}}.
\end{equation}

Assume that $|S|\leq 2k-\lambda-3=3\cdot 2^{2r-3}-3\cdot 2^{r-2}-3$. It follows that $a+b=v-|S|\geq 2^{2r-3}+2^{r-2}+3$. As $a\geq 4$, we deduce that $ab\geq 4(2^{2r-3}+2^{r-2}-1)$. Using \eqref{abh}, we obtain $|S|\geq \frac{4(2^{2r-3}+2^{r-2}-1)(2^{r-2}-1)}{2^{r-2}+2^{r-5}}>3\cdot 2^{2r-3}-3\cdot 2^{r-2}-3$,
where the last inequality follows from $r\geq 3$ and some straightforward calculations. This contradiction finishes our proof.\end{proof}


\section{The elliptic quadric graphs $O^{-}(2r,2)$}\label{sec:elliptic}

The elliptic quadric graph $O^{-}(2r,2)$ is the subgraph of $Sp(2r,2)$ induced by $V^{-}:=\{(x_1,\dots,x_{2r})^t\in \mathbb{F}_2^{2r}:x_1^2+x_2^2+x_1x_2+x_3x_4+\dots +x_{2r-1}x_{2r}=1\}$ (the complement of an elliptic quadric in $\mathbb{F}_2^{2r}$). The vertex $x:=(x_1,\dots,x_{2r})^t$ is adjacent to $y:=(y_1,\dots,y_{2r})^t$  if $x^tMy=1$, where $M$ is defined in \eqref{M}. It is known (see \cite{GR2}) that $O^{-}(2r,2)$ is a $(2^{2r-1}+2^{r-1},2^{2r-2}+2^{r-1},2^{2r-3}+2^{r-2},2^{2r-3}+2^{r-1})$-SRG. The value of $2k-\lambda-2$ is $3\cdot(2^{2r-3}+2^{r-2})-2$.

The following proposition is the main result of this section.
\begin{prop}
For $r\geq 3$, $\kappa_2(O^{-}(2r,2))=3(2^{2r-3}+2^{r-2})-3=2k-\lambda-3$.
\end{prop}
\begin{proof}
By the remarks in Section \ref{copolarsec} involving Hall's characterization of copolar spaces, we know that $O^{-}(2r,2)$ is a counterexample to Brouwer's Conjecture and that a good candidate for $\kappa_2(O^{-}(2r,2))$ is the size of the neighborhood of a clique of $O^{-}(2r,2)$ corresponding to a hyperbolic line. Such a clique has the form of $C=\{x,y,x+y\}$ where $x$ and $y$ are adjacent. By the remarks in Section \ref{copolarsec}, it follows that $|N(C)|=2k-\lambda-3=3(2^{2r-3}+2^{r-2})-3$.

The fact that $O^{-}(2r,2)\setminus (C\cup N(C))$ has no singleton components follows from Section \ref{copolarsec}, but can be also proved using the last part of Lemma \ref{lowerbounddiscset}. Note that $C$ induces a clique of size $3$ in $O^{-}(2r,2)$ and $N(C)=2k-\lambda-3$ so the last part of Lemma \ref{lowerbounddiscset} can be applied here. 
As $\mu/k=(2^{2r-3}+2^{r-1})/(2^{2r-2}+2^{r-1})<\tfrac{2}{3}$, it follows by Lemma \ref{lowerbounddiscset} that $O^{-}(2r,2)\setminus N(\{x,y,x+y\})$ does not contain any singleton components. This proves $\kappa_2(O^{-}(2r,2))\leq 3(3^{2r-3}+2^{r-2})-3=2k-\lambda-3$.

To show that $\kappa_2(O^{-}(2r,2))\geq 2k-\lambda-3$, we will prove that if $A$ induces a $K_{1,2}$ or a connected subgraph of order at least $4$, then $|S|\geq 2k-\lambda-3$.

If $A$ induces a $K_{1,2}$, then Lemma \ref{lowerbounddiscset} implies that $|S|\geq 2k-\lambda-3$.

If $|A|\geq 4$, then we will show that $|S|\geq 2k-\lambda-2$ when $r\geq 4$ and $|S|\geq 2k-\lambda-3$ when $r=3$.

When $r\geq 4$, Lemma \ref{HL1} implies
\begin{equation}\label{abh2}
|S|\geq \frac{4ab\mu}{(\lambda-\mu)^2+4(k-\mu)}=\frac{ab(2^{r-2}+1)}{2^{r-2}+2^{r-5}}.
\end{equation}

Assume that $|S|\leq 2k-\lambda-3=3(2^{2r-3}+2^{r-2})-3$. It follows that $a+b=v-s\geq 2^{2r-3}-2^{r-2}+3$. As $a\geq 4$, we deduce that $ab\geq 4(2^{2r-3}-2^{r-2}-1)$. Using \eqref{abh2}, we obtain $|S|\geq \frac{4(2^{2r-3}-2^{r-2}-1)(2^{r-2}+1)}{2^{r-2}+2^{r-5}}>3(2^{2r-3}+2^{r-2}-1)$, where the last inequality follows from $r\geq 4$ and some straightforward calculations. This gives us a contradiction and finishes the proof of this case. 

When $r=3$, then $v=36,k=20,\lambda=10,\mu=12$ and we will show that $|S|\geq 2k-\lambda-3=27$. Assume $|S|\leq 26$ which implies $a+b=v-|S|\geq 10$. As $a\geq 4$, this means $ab\geq 24$. Using Lemma \ref{HL1}, we obtain $|S|\geq \frac{4ab}{3}\geq 32$ which is a contradiction with $|S|\leq 26$. This finishes our proof.
\end{proof}

We remark that there are exactly $32548$ non-isomorphic $(36,20,10,12)$-SRGs, as shown by McKay and Spence \cite{McKS}. We leave as an open problem the characterization of disconnecting sets of size $2k-\lambda-3$ in $O^{-}(2r,2)$ whose removal disconnects the graph into non-singleton components.

\section{The lattice graphs $L_2(n)=K_n\times K_n$ are OK}\label{sec:lattice}

The lattice graph $L_2(n)$ (also called Hamming graph or $OA(2,n)$ strongly regular graph; see \cite{BCN,BH1}) is the line graph of the complete bipartite graph $K_{n,n}$;
its vertex set is the cartesian product $[n]\times [n]=\{ab:1\leq a,b\leq n\}$ and $ab\sim xy$ if and only if $a=x$ or $b=y$. It is easy to show that $L_2(n)$ is a $(n^2,2(n-1),n-2,2)$-SRG and it is actually known that for $n\neq 4$, any $(n^2,2(n-1),n-2,2)$-SRG is isomorphic to $L_2(n)$ (see \cite{BH1,VW}). When $n=4$, there exists a $(16,6,2,2)$-SRG called the Shrikhande graph that is not isomorphic to $L_2(4)$ (see \cite{BH1} or Section 9 where we prove this graph is OK).

\begin{lem}\label{lattice}
For $n\geq 3$, $\kappa_2(L_2(n))=2k-\lambda-2=3n-4$. The only disconnecting sets of size $3n-4$ are $N(\{u,v\})$ where $u,v\in V(L_2(n))$ are adjacent and (modulo a permutation of the first and second coordinates) $\{13,14,23,24,32,31,42,41\}$ for $n=4$.
\end{lem}
\begin{proof}
When $n=3$, we have $3n-4=5$. Because $9=5+2+2$, the only disconnecting subsets of order $5$ are of the form $N(\{u,v\})$ where $u,v$ are adjacent.
Since the degree of $L_2(3)$ is $4$, it follows from \cite{BM} that $\kappa_2(L_2(3))\geq 5$. This proves the result for $n=3$.

When $n=4$, we have $3n-4=8$. Removing the subset of vertices
\begin{equation}\label{indL2_4}
\{13,14,23,24,32,31,42,41\}
\end{equation}
will disconnect $L_2(4)$ into two components whose vertex sets are $\{11, 12,$ $21, 22\}$ and $\{33, 34,$ $ 43, 44\}$ respectively. By deleting a disconnecting set of $8$ vertices of $L_2(4)$, we obtain a disconnected graph on $8$ vertices that will contain one component of at most $4$ vertices. If this component has $1$ or $2$ vertices, then we are done. If this component has $3$ vertices, then it is either $K_3$ or $K_{1,2}$. In the first case, we deduce that the disconnecting set has $1+3\cdot 3=10$ vertices which is a contradiction. In the other case, the disconnecting set has at least $9$ vertices which is again a contradiction. If this component has $4$ vertices, then there exists exactly one other component also of $4$ vertices. A connected subgraph of $L_2(4)$ with $4$ vertices is $K_4, C_4, P_4$ or $K_3$ with a pendant edge. By a case analysis, the only way this can happen is if the disconnecting set is (modulo some coordinate permutation) as in \eqref{indL2_4}.

For the rest of the proof, we assume $n\geq 5$. If $A$ induces a clique of size $a\geq 2$, then without loss of generality we may assume that $A=\{ 11, \dots, 1a\}$. We obtain $|N(A)|=n-a+a(n-1)\geq 3n-4$ with equality if and only if $a=2$. Thus, $|S|\geq |N(A)|\geq 3n-4$ with equality if and only if $S$ is the neighborhood of an edge.

If $A$ is not a clique, then we may also assume that any component of $B$ is not a clique (as otherwise we can repeat the argument from the previous case). Each of $A$ and $B$ will contain two non-adjacent vertices. By permuting the first and the second coordinates, we may assume that $\{12,21\}\subset A$ and $\{34,43\}\subset B$. The following are
(interior) vertex-disjoint paths with one endpoint in $\{12,21\}$ and the other endpoint in $\{34,43\}$:
\begin{align*}
Q_1&=\{\{12,1x,3x,34\}:x\geq 5\}\\
Q_2&=\{\{21,x1,x3,43\}:x\geq 5\}\\
Q_3&=\{\{12,y2,y4,34\}:y\geq 5\}\\
Q_4&=\{\{21,2y,4y,43\}:y\geq 5\}\\
Q_5&=\{\{ab,cb,cd\}, \{ab,ac,dc\} :ab\in \{12,21\}, cd, dc \in \{34,43\} \}.
\end{align*}
There are $n-4$ paths in each $Q_i$ for $1\leq i \leq 4$ and there are $8$ paths in $Q_5$. Hence, we have found $4(n-4)+8=4n-8$ (interior) vertex-disjoint paths between $\{12,21\}$ and $\{34,43\}$. This implies $|S|\geq 4n-8>3n-4$ and finishes our proof.
\end{proof}

\section{The $OA(3,n)$ Latin square graphs are OK}\label{sec:oa}

An orthogonal array $OA(t,n)$ with parameters $t$ and $n$ is a $t\times n^2$ matrix with entries from the set $[n]=\{1,\dots,n\}$ such that the $n^2$ ordered pairs defined by any two distinct rows of the matrix are all distinct. An orthogonal array $OA(t,n)$ is equivalent to $t-2$ mutually orthogonal Latin squares (see \cite{BH1} or \cite[Section 10.4]{GR}). Thus, an $OA(3,n)$ is equivalent to a Latin square of order $n$. These are known to exist for any $n\geq 2$ and can be regarded as a generalization of groups as they are equivalent to the multiplication table (Cayley table) of a quasigroup on $n$ elements (see \cite[Chapter 17]{VW}).

Given an orthogonal array $OA(t,n)$, one can define a graph $G$ as follows: the vertices of $G$ are the $n^2$ columns of the orthogonal array and two vertices are adjacent if they have the same entry in one coordinate position. It is known that $G$ is a $(n^2,t(n-1),n-2+(t-1)(t-2),t(t-1))$-SRG (see \cite[Section 10.4]{GR}). It is easy to see that the graph associated to an $OA(2,n)$ is isomorphic to the lattice graph $L_2(n)$. As an $OA(3,n)$ is equivalent to a Latin square of order $n$, a graph obtained from an $OA(3,n)$ orthogonal array is called a Latin square graph and is a $(n^2,3(n-1),n,6)$-SRG. Bruck \cite{Bruck} showed a partial converse to the previous statement by proving that when $n>23$, any $(n^2,3(n-1),n,6)$-SRG is a Latin square graph.

\begin{lem}\label{oa3n} For any $n\geq 4$, if $G$ is the $(n^2,3(n-1),n,6)$-SRG associated with an $OA(3,n)$, then
\begin{equation}
\kappa_2(G)=2k-\lambda-2=5n-8.
\end{equation}
The disconnecting sets of size $5n-8$ are of the form $N(\{u,v\})$ where $u$ and $v$ are two adjacent vertices in $G$ or $N(A)$ where $A=\{[1,x_i,y_i]^t\}_{1\leq i\leq 4}$ induces a clique of order $4$.
\end{lem}
\begin{proof}
If $n=4$, then the graph $G$ is an $(16,9,4,6)$-SRG which satisfies the conjecture by Lemma \ref{smallv}.

Assume $n\geq 5$ for the rest of the proof. We will show that $|S|\geq 2k-\lambda-2$ with equality if and only if $S$ is as described in the Lemma \ref{oa3n}.

We have two cases:



1. Either $A$ or $B$ induces a clique in $G$.

Without any loss of generality assume that $A$ is a clique. If $|A|=2$, then there is nothing to prove.

Assume that $|A|=r\geq 3$.

If $r=3$, then without loss of generality we have two possible situations:


a) $A=\{[1,x_i,y_i]^t\}_{1\leq i\leq 3}$, where $x_i\neq x_j$ and $y_i\neq y_j$ for $1\leq i<j\leq 3$.

In this case, the common neighbors of the vertices in $A$ are the vertices of the form $[1,u,v]^t$ where $u\in [n]\setminus \{x_1,x_2,x_3\}$ and there are $n-3$ such vertices. Inclusion and exclusion and $n\geq 5$ imply that
\begin{equation*}
|S|\geq |N(A)|=3(k-2)-3(\lambda-1)+(n-3)=7n-15>5n-8.
\end{equation*}


b) $A=\{[1,x_1,y_1]^t, [1,x_2,y_2]^t,[2,x_1,y_2]^t\}$, where $x_1\neq x_2$ and $y_1\neq y_2$.

In this case, the three vertices of $A$ could have at most one common neighbor $[2,x_2,y_1]^t$ (this happens if the only vertex of $G$ whose first two coordinates are $2$ and $x_2$ is $[2,x_2,y_1]^t$). Inclusion and exclusion and $n\geq 5$ imply that
\begin{equation*}
|S|\geq |N(A)|\geq 3(k-2)-3(\lambda-1)=6n-12>5n-8.
\end{equation*}

If $r\geq 4$, then without loss of generality
\begin{equation}\label{oa1}
A=\{[1,x_1,y_1]^t, [1,x_2,y_2]^t, [1,x_3,y_3]^t,\dots, [1,x_r,y_r]^t\},
\end{equation}
where $x_i\neq x_j$ and $y_i\neq y_j$ for all $1\leq i<j\leq r$ or $r=4$ and
\begin{equation}\label{oa2}
A=\{[1,x_1,y_1]^t, [1,x_2,y_2]^t, [2,x_1,y_2]^t, [2,x_2,y_1]^t\},
\end{equation}
where $x_1\neq x_2$ and $y_1\neq y_2$. Note that the second situation may or may not happen.

We will use the following notation for the rest of the proof. Given two elements $x,y\in [n]$, $[x,y,*]^t$ will denote the vertex of $G$ whose 1st entry is $x$ and 2nd entry is $y$. We define $[x,*,z]^t$ and $[*,y,z]^t$ similarly.

When $A$ is the set given in \eqref{oa1}, the set $N(A)$ will consist of the following vertices:
\begin{align*}
[1,\beta,*]^t &, \quad \beta \neq x_i,\quad  \forall i, 1\leq i\leq r\\
[\ast,x_i,\epsilon]^t &, \quad \epsilon\neq y_j, \quad  \forall i,j,  1\leq i,j\leq r\\
[*, x_i, y_j]^t &, \quad \forall i,j, \quad 1\leq i\neq j\leq r.
\end{align*}
This implies $|N(A)|=(n-r)(r+1)+r(r-1)=(r+1)n-2r$. If $r\geq 5$, we get $|S|\geq |N(A)|\geq 6n-10>5n-8$ as $n\geq 5$.  If $r=4$, then $|S|\geq |N(A)|=5n-8=2k-\lambda-2$ with equality if and only if $S=N(A)$.

When $A$ is the set given in \eqref{oa2}, then any three distinct vertices of $A$ will have no common neighbors in $N(A)$. Inclusion and exclusion and $n\geq 5$ imply that
\begin{align*}
|S|&\geq |N(A)|=4(k-3)-6(\lambda-2)=6n-12>5n-8.
\end{align*}


2. Both $A$ and $B$ do not induce a clique in $G$.

In this case, $A$ must contain two non-adjacent vertices $[x_1,x_2,x_3]^t$ and $[y_1,y_2,y_3]^t$ and $B$ must contain two non-adjacent vertices $[z_1,z_2,z_3]^t$ and $[w_1,w_2,w_3]^t$. Because there are no edges between $A$ and $B$, it follows that these four vertices are pairwise non-adjacent and thus, $x_i,y_i,z_i$ and $w_i$ are distinct for every $i\in \{1,2,3\}$.

The following are (interior) vertex-disjoint paths of length $3$ from $[x_1,x_2,x_3]^t$ to $[z_1,z_2,z_3]^t$:
\begin{align}
[x_1,x_2,x_3]^t, [x_1,u,*]^t, & [z_1,u,*]^t, [z_1,z_2,z_3]^t, \forall u\in [n]\setminus \{x_2,y_2,z_2,w_2\} \\
[x_1,x_2,x_3]^t, [*,x_2,v]^t, & [*,z_2,v]^t, [z_1,z_2,z_3]^t, \forall v\in [n]\setminus \{x_3,y_3,z_3,w_3\} \\
[x_1,x_2,x_3]^t, [s,*,x_3]^t, & [s,*,z_3]^t, [z_1,z_2,z_3]^t, \forall s\in [n]\setminus \{x_1,y_1,z_1,w_1\}
\end{align}

The following are (interior) vertex-disjoint paths of length $3$ from $[y_1,y_2,y_3]^t$ to $[w_1,w_2,w_3]^t$:
\begin{align}
[y_1,y_2,y_3]^t, [y_1,u,*]^t, & [w_1,u,*]^t, [w_1,w_2,w_3]^t, \forall u\in [n]\setminus \{x_2,y_2,z_2,w_2\} \\
[y_1,y_2,y_3]^t, [*,y_2,v]^t, & [*,w_2,v]^t, [w_1,w_2,w_3]^t, \forall v\in [n]\setminus \{x_3,y_3,z_3,w_3\} \\
[y_1,y_2,y_3]^t, [s,*,y_3]^t, & [s,*,w_3]^t, [w_1,w_2,w_3]^t, \forall s\in [n]\setminus \{x_1,y_1,z_1,w_1\}
\end{align}

The following are (interior) vertex-disjoint paths of length $2$ between $[x_1,x_2,x_3]^t$ to $[z_1,z_2,z_3]^t$. To simplify our notation, we only list the middle vertex of each path:
\begin{equation*}
[x_1,z_2,*]^t; [z_1,x_2,*]^t; [x_1,*,z_3]^t; [z_1,*,x_3]^t; [*,x_2,z_3]^t;[*,z_2,x_3]^t.
\end{equation*}
The following are (interior) vertex-disjoint paths of length $2$ between $[y_1,y_2,y_3]^t$ and $[w_1,w_2,w_3]^t$. Again, we only list the middle vertex of each path:
\begin{equation*}
[y_1,w_2,*]^t; [w_1,y_2,*]^t; [y_1,*,w_3]^t; [w_1,*,y_3]^t; [*,y_2,w_3]^t; [*,w_2,y_3]^t.
\end{equation*}

Hence, there are at least $2\cdot 3\cdot (n-4)+2\cdot 6=6n-12$ interior vertex-disjoint paths between $A$ and $B$. As $n\geq 5$, this implies $|S|\geq 6n-12>5n-8$ and finishes our proof.
\end{proof}

\section{Primitive strongly regular graphs with at most 30 vertices}\label{sec:primitive}

In general, Brouwer's Conjecture is false. In this section we check which parameters of small strongly regular graphs from the list \cite{BH2}
satisfy the conjecture. 

The following useful property is an immediate consequence of Cauchy's interlacing theorem (see \cite{BH1,GR,H1} for more details).
\begin{property}\label{property}
Let $\alpha$ and $\beta$ be the largest eigenvalues of the subgraphs of $G$ induced by $A$ and $B$, respectively. Then $min(\alpha, \beta) \leq \theta_2(G)$.
\end{property}
\begin{proof}
By Cauchy's interlacing theorem, the eigenvalues of the subgraph of $G$ induced by $A\cup B$ interlace the eigenvalues of $G$. Thus, $\theta_2(G)$ is at least the second largest eigenvalue of the subgraph induced by $A\cup B$ which is at least $\min(\alpha, \beta)$ (as $A$ and $B$ are not connected by any edges).
\end{proof}

The previous property enables us to show that Brouwer's Conjecture is true when $\theta_2$ is very small.
\begin{prop}\label{conchang}
If $G$ is a connected SRG such that the second largest eigenvalue $\theta_2 < \sqrt{2}$, then $G$ is OK.
\end{prop}
\begin{proof}
Suppose that $G$ is not OK. Assume $V(G)=A\cup S\cup B$, where $S$ is a disconnecting set, $N(A)\subset S$, $B=V(G)\setminus (A \cup S), |A| \geq 3$ and each component of $B$ has at least $3$ vertices. It follows that each of $A$ and $B$ has a clique of order $3$ or a path with $3$ vertices as an induced subgraph. The largest eigenvalues of a clique of size $3$ and a path with $3$ vertices are $2$ and $\sqrt{2}$, respectively. Therefore, by Property \ref{property} we obtain $\theta_2 \geq \sqrt{2}$, a contradiction.
\end{proof}

\begin{example}[$(26,15,8,9)$-SRGs]\label{26graph}
Let $\Gamma$ be a $(26,15,8,9)$-SRG (There are exactly 10 of these graphs \cite{BH2}).
Let $A$ be a subset of vertices that induces a connected graph such that $V\setminus S$ is the disjoint union of $A$ and $B$ with $|B|\geq |A| \geq 3$.
Since $2k-\lambda-2 = 20$, $\theta_2 = 2$, and the complement of $(26,15,8,9)$-SRG has $\overline{\lambda} = 3$ and $\overline{\mu} = 4$, we may assume that we are in one of the following two situations:
\begin{enumerate}
\item $A$ and $B$ such that $|A|=|B|=4$.
Then $A$ and $B$ are cliques as otherwise $\overline{\lambda} \geq 4$. But by Property \ref{property}, this is a contradiction with $\theta_2 = 2$. One can also use Lemma \ref{haemersbnd} to obtain a contradiction in this case.
\item $A$ and $B$ such that $|A| = 3$ and $|B| \geq 4$.
Then $A$ is a triangle and $|B| = 4$ as $\overline{\lambda} = 3$ and $\overline{\mu} = 4$.
\end{enumerate}

First, we show that the induced subgraph $B$ is a cycle on $4$ vertices $C_4$.
Since by inclusion and exclusion, $19 = |N(A)| = 3(k-2) -3(\lambda -1) +$ (the number of common neighbors of $A$) and $3k-3\lambda-3=18$, there exists exactly one common neighbor of $A$, say $d$.
Since $\lambda = 8$,
\begin{enumerate}
\item[($\ast$)] $|N(a) \cap N(a') \cap N(A) \setminus \{d\} | = 6$ for distinct $a, a' \in A$, i.e., $|A \cap N(c)| = 2$ for each $c \in N(A) \setminus \{d\}$.
\end{enumerate}
Fix a vertex $b \in B$. Since $\mu = 9$, there are $|A|\cdot \mu=27$ paths of length $2$ between $b$ and $A$. By ($\ast$), $b$ is adjacent to $d$ and $|N(A) \cap N(b)| = 13$. This implies that $B$ induces a $C_4$ and $d$ is adjacent to all vertices of $B$.

Next, we consider $C:= N(A) \cap N(d)$. Since $\lambda = 8$, $|N(a) \cap C|=6$ for each $a \in A$. This implies that the number of common neighbors of $d$ and all of vertices of $A$ is at least $2$, a contradiction. Therefore, $\Gamma$ is OK.
\end{example}

\begin{example}[The Schl\"{a}fli graph]\label{sch}
Let $\Gamma$ be a $(27, 16, 10, 8)$-SRG. Seidel \cite{S} has shown that there is a unique strongly regular graph with these parameters and for each vertex $w \in V(\Gamma)$, the subgraph induced by $N_1(w)$ is the halved $5$-cube.

Let $C$ be a subset of $\Gamma$.
\begin{enumerate}
\item[($\ast$)] If $C$ is a triangle, then since the halved 5-cube has $\lambda=6$,
the inclusion and exclusion principle yields $|N(C)|=3(k-2)-3(\lambda-1) + 6 = 21$.
\item[($\ast \ast$)] If $C$ is a path of length $2$, then since the halved $5$-cube has $\mu=6$,
the inclusion and exclusion principle yields $|N(C)|= 2(k-1) + (k-2) - 2\lambda - (\mu-1) + 6 = 23$.
\end{enumerate}

Let $A$ be a subset with at least $3$ vertices. Assume that the subgraph induced by $A$ is connected and $B:= V \setminus (A \cup S)$ satisfies $|A| \leq |B|$.
Then $A$ contains a triangle or a path of length $2$. Now by ($\ast$) and ($\ast \ast$), $|S| + |A| \geq 24$.
So $|B| \leq 2$, a contradiction. Therefore, the Schl\"{a}fli graph is OK.
\end{example}

In the next table, we report the strongly regular graphs with at most $30$ vertices. The symbol $\circ$ means `For all examples, Brouwer's Conjecture is true'. The symbol $\times$ means `There exists at least one counterexample'.

\medskip

\begin{tabular}{p{3cm}ccccccccc}
\hline
 No. & $v$      &   $k$     &     $\lambda$    &      $\mu$       &    $r^f$    & $s^g$     &  Brouwer's Conjecture   &     Comment \\
\hline
1    &     5       &     2     &  0            &      1        &     $0.618^2$   &    $-1.618^2$     &     $\circ$     &  L \ref{smallv}   \\
2    &    9       &     4    &  1              &     2        &      $1^4$      &    $ -2^4$       &     $\circ$     &  L \ref{smallv}   \\
3    &    10       &     3    &  0              &     1        &      $1^5$      &    $ -2^4$       &     $\circ$     &  P \ref{kSRG}   \\
$\overline{3}$    &    10       &     6    &  3              &     4        &      $1^4$      &    $ -2^5$       &     $\circ$     &  $\kappa_2(G)=\infty$   \\
4    &    13       &     6    &  2              &     3        &      $1.303^6$      &    $ -2.303^6$       &     $\circ$   &  L \ref{diff}   \\
5    &    15       &     6    &  1              &     3        &      $1^9$      &    $ -3^5$       &     $\circ$     &  P \ref{conchang}   \\
$\overline{5}$    &    15       &     8    &  4              &     4        &      $2^5$      &    $ -2^9$       &     $\times$     &  P \ref{triangular} \\

6    &    16       &     5    &  0              &     2        &      $1^{10}$      &    $ -3^5$       &     $\circ$   & P \ref{kSRG}   \\
$\overline{6}$    &    16       &     10    &  6              &     6        &      $2^5$      &    $ -2^{10}$       &     $\circ$   &  L \ref{smallv}   \\
7    &    16       &     6    &  2              &     2        &      $2^6$      &    $ -2^9$       &     $\circ$   &  P \ref{kSRG}   \\
$\overline{7}$   &    16       &     9    &  4              &     6        &      $1^9$      &    $ -3^6$       &     $\circ$   &  L \ref{smallv}   \\
8   &     17       &     8     &  3            &      4        &     $1.562^8$   &    $-2.562^8$     &     $\circ$     &  L \ref{diff}   \\
9    &     21       &     10     &  3            &      6        &     $1^{14}$   &    $-4^6$     &     $\circ$     &  P \ref{conchang}   \\
$\overline{9}$   &     21       &     10     &  5            &      4        &     $3^6$   &    $-2^{14}$     &     $\times$     &  P \ref{triangular}   \\

10    &     25       &     8     &  3           &      2       &     $3^8$   &    $-2^{16}$     &     $\circ$     &  L \ref{lattice}   \\
$\overline{10}$    &     25       &     16     &  9      &      12        &     $1^{16}$   &    $-4^8$     &     $\circ$     &  L \ref{smallv}   \\
11    &     25       &     12     &  5     &     6     &     $2^{12}$   &    $-3^{12}$     &     $\circ$     &  L \ref{diff}   \\

12    &     26       &     10     &  3     &     4     &     $2^{13}$   &    $-3^{12}$     &   $\circ$   &  L \ref{diff} \\
$\overline{12}$    &     26       &     15     &  8     &     9     &     $2^{12}$   &    $-3^{13}$     &  $\circ$        &  E \ref{26graph} \\
13    &     27       &     10     &  1     &     5     &     $1^{20}$   &    $-5^6$     &   $\circ$       &  P \ref{conchang} \\
$\overline{13}$    &     27       &     16     &  10     &     8     &     $4^6$   &    $-2^{20}$     &  $\circ$        &  E \ref{sch}   \\
14    &     28       &     12     &  6     &     4     &     $4^7$   &    $-2^{20}$     &   $\times$        &  P \ref{triangular}    \\
$\overline{14}$    &     28       &     15     &  6     &     10     &     $1^{20}$   &    $-5^7$     &  $\circ$       & P \ref{conchang}   \\
15    &     29       &     14     &  6     &     7     &     $2.193^{12}$   &    $-3.193^{14}$     &   $\circ$       &  L \ref{diff}  \\
\hline
\end{tabular}

\medskip

In order to prove that the three Chang graphs (these are $(28,12,6,4)$-SRGs which are not $T(8)$) are OK, we use some work of Delsarte including the notion of a Delsarte clique which we briefly describe below.

Delsarte \cite[p. 31]{del} obtained a linear programming bound for cliques in
strongly regular graphs. It was observed by Godsil \cite[p. 276]{Godsilac} that
the same bound holds for distance-regular graphs, as follows.

\begin{prop}\label{delbound} Let $\Gamma$ be a distance-regular graph with
valency $k$ and smallest eigenvalue $\theta_{\min}$. If $C$ is a clique in $\Gamma$ with $c$ vertices, then $c \leq 1+\frac{k}{-\theta_{\min}}.$
\end{prop}

\noindent A clique $C$ in a distance-regular graph $\Gamma$ that attains the
above bound is called a {\em Delsarte clique}. Lemmas 13.7.2 and 13.7.4  in \cite{Godsilac}
characterize such cliques.

\begin{example}[$(28,12,6,4)$-SRGs]\label{28graph}
Let $\Gamma$ be a $(28,12,6,4)$-SRG (There are exactly 4 of these graphs \cite{BH2}, namely $T(8)$ and the three Chang graphs.).
Suppose that $\Gamma$ is a counterexample of Brouwer's Conjecture. By Lemma \ref{HL1} and $2k -\lambda -2 = 16$,  $\frac{4ab}{9} \leq |S| \leq 15$.
Then $ab \leq 33$ and the only integral solutions with $3 \leq a \leq b$ and $a + b \geq 13$ are $a = 3$ and $b =10, 11$. But $b \leq \overline{\mu}=10$.
So $a = 3$ and $b =10$. Now as $\overline{\lambda} = 6$ we obtain that $A$ is a triangle, and we see that every vertex in $S$ is adjacent to exactly two vertices in $A$.

For $x \in A$, let $S_x = \{ w \in S \mid w \not\sim x\}$. As $\lambda=10$, we deduce that $|S_x| = 5$.  As $x\in A$ and $w\in S_x$ have $4$ common neighbors
with two of these common neighbors being in $A$,  and $x\in A$ and $w'\in S\setminus S_x$ have $6$ common neighbors with one of these common neighbors being in $A$,
we deduce that each $S_x$ is a clique of order $5$ for $x\in A$. Also, we deduce that $w\in S_x$ has exactly two neighbors in $S\setminus S_x$.

By similar arguments, considering a vertex $x\in A$, a vertex $u\in B$ and using $\mu=4$, we deduce that every vertex $u\in B$ has exactly two neighbors in $S_x$ and thus, has exactly $6$ neighbors in $S$.

Next, we consider the partition $\pi = A \cup S \cup B$ of $V(\Gamma)$. From the previous arguments, we deduce that $\pi$ is an equitable partition (see \cite[Chapter 9]{GR}) whose quotient matrix is the following:
\begin{equation*}
Q=
\begin{bmatrix}
2 & 10 & 0   \\
2 & 6 & 4   \\
0 & 6 & 6   \\
\end{bmatrix}.
\end{equation*}

We claim that $S$ is the $3 \times 5$ grid and $B$ is the triangular graph $T(5)$.

By Cauchy's interlacing theorem, the eigenvalues of the subgraph of $\Gamma$ induced by $B$ interlace the eigenvalues of $\Gamma$. In particular, smallest eigenvalue of $B$ is at least $-2$. By Theorem 3.12.2(i) of \cite{BCN}, $B$ is isomorphic to the line graph of a regular or bipartite semiregular connected graph $\Delta$. Since $B$ is a $6$-regular graph of order $10$, it is easily checked that $\Delta$ is a complete graph of order $5$. So $B$ is the line graph of $\Delta$, that is  $T(5)$.

For $w \in S$, we consider the set $N(w) \cap B$ of size $4$. As $T(5)$ has $\mu =4$, we see that $N(w) \cap B$ is a clique of order $4$,  and each $w \in S$ corresponds to such a clique. Now $T(5)$ has exactly 5 such cliques.

The vertices in $S$ corresponding to the same clique of order $4$ in $T(5)$ are adjacent since $\mu = 4$. It follows that they together form a clique $C$ of order 7 as the Delsarte bound is $7$, and hence they are a Delsarte clique. This implies that any vertex outside $C$ are adjacent to exactly two neighbors in $C$.
It follows that $S$ is the $3 \times 5$-grid and any maximal triangle corresponds to a maximal clique in $T(5)$. It follows that $\Gamma$ is the triangular graph $T(8)$.
\end{example}

The main part of the proof of Brouwer and Mesner \cite{BM} is to show that their result holds for the case of strongly regular graphs with smallest eigenvalue $-2$. In the next proposition, we discuss the strongly regular graphs with smallest eigenvalue $-2$.

\begin{prop}\label{sei}
Among the strongly regular graphs with smallest eigenvalue $-2$, the only counterexamples of Brouwer's Conjecture are the triangular graphs $T(m)$, where $m \geq 6$.
\end{prop}
\begin{proof}
By Seidel's classification (see \cite{BH1,BM,S}), a strongly regular graph with $\theta_v = -2$ is one of the following:
the complement of the ladder graph, a lattice graph, the Shrikhande graph, a triangular graph, one of the three Chang graphs, the Petersen graph, the Clebsch graph and the Schl\"{a}fli graph.
By Lemma \ref{smallv}, the complement of the ladder graph is OK.
By Lemma \ref{lattice}, a lattice graph is OK.
By Lemma \ref{triangular}, the triangular graphs $T(m)$ is not OK, where $m \geq 6$.
By the above table and Example \ref{28graph}, one of the three Chang graphs, the Petersen graph, the Clebsch graph and the Schl\"{a}fli graph are OK.
This completes the proof of the proposition.
\end{proof}

\section{Final Remarks}\label{sec:final}

In this paper, we have shown that Brouwer's Conjecture is false in general by showing that strongly regular graphs obtained from copolar spaces or $\Delta$-spaces form infinite families of counterexamples. It would be interesting to determine the best general lower bound for $\kappa_2(G)$ when $G$ is a strongly regular graph. Note that the parameter $\kappa_2(G)$ of a strongly regular graph $G$ does not only depend on the parameters of $G$, but also on its structure, as the triangular graph $T(8)$ and the three Chang graphs $C_1,C_2, C_3$ (which are all $(28,12,6,4)$-SRGs), show: $\kappa_2(T(8))=15<16=\kappa_2(C_i)$ for each $1\leq i\leq 3$. 

The symplectic graphs $Sp(2r,q)$ over $\mathbb{F}_q$ show that the gap between the connectivity conjectured by Brouwer and the actual connectivity can be arbitrarily large. For the other three counterexamples coming from copolar spaces: the triangular graphs $T(m)$, the hyperbolic quadric graphs $O^{+}(2r,2)$ and the elliptic quadric graphs $O^{-}(2r,2)$, this gap is exactly $1$. For all of these counterexamples $G$, the value of $\kappa_2(G)$ equals the size of the neighborhood of a clique (corresponding to a hyperbolic line of the space). It would be interesting to see if for every counterexample, the minimum disconnecting set is the neighborhood of a clique. Although Brouwer's Conjecture is false, we believe it is interesting problem to classify to find the value of $\kappa_2(G)$ for other strongly regular graphs. In view of Proposition \ref{kSRG}, we believe Brouwer's Conjecture is true for all $(v,k,\lambda,\mu)$-SRGs having $k\geq 2\lambda+1$.

As mentioned in the first section, Brouwer and Koolen \cite{BK} have recently proved that the vertex-connectivity of a distance-regular graph of degree $k$ equals $k$. They have also proved that the only disconnecting sets of size $k$ are the neighborhoods of the vertices of the graph. In view of these results, we believe that investigating the value of $\kappa_2(G)$ when $G$ is a distance-regular graph, is an interesting project. As observed by Brouwer and Koolen \cite{BK}, the icosahedron graph which is a distance-regular graph of degree $5$ and order $12$ with intersection array $\{5,2,1;1,2,5\}$ (see \cite{BCN} for more details) can be disconnected into two triangles by removing the $k+1=6$ vertices of a hexagon. In this case, $2k-\lambda-2=10-2-2=6$. Also, the line graph of the Petersen graph which is a distance-regular graph of degree $4$ and order $15$ with intersection array $\{4,2,1;1,1,4\}$ (see \cite{BCN} for more details) can be disconnected into two pentagons by removing an independent set of size $k+1=5$. In this case, $2k-\lambda-2=8-1-2=5$.

Another problem that deserves further exploration is determining the vertex-connectivity of the second subconstituents of strongly regular graphs. If $x$ is a vertex of a $(v,k,\lambda,\mu)$-SRG $G$, then the second subconstituent $G_2(x)$ of $G$ with respect to $x$ is the subgraph of $G$ induced by the vertices at distance exactly $2$ from $x$. It is easy to see that $G_2(x)$ is a $(k-\mu)$-regular graph and it is known that if $G$ is not a complete multipartite graph, then $G_2(x)$ is connected (see \cite{BH1} for an eigenvalue proof of this fact). As observed by Brouwer and Haemers  (see \cite{BH1} p.125 in our version), there are strongly regular graphs $G$ and vertices $x$ of such graphs with the property that the vertex-connectivity of $G_2(x)$ is less than $k-\mu$. The example provided by Brouwer and Haemers is a $(96,76,60,60)$-SRG (the complement of this graph Haemers(4) was constructed by Haemers in his Ph.D. Thesis; see also \cite{BL} $\S$8A) which has $k-\mu=16$ and the second subconstituent of every vertex has vertex-connectivity $15$.  If one could find $(v,k,\lambda,\mu)$-SRGs where the second subconstituent has connectivity less than $k-\mu$ and $\mu>\lambda+2$, then such graph would be a counterexample to the Brouwer's Conjecture (one could obtain a disconnecting set of size less than $2k-\mu<2k-\lambda-2$ by taking the union of $x$, its neighbors and a disconnecting set of $G_2(x)$ of size less than $k-\mu$).

\section*{Acknowledgements}

The authors are grateful to Jonathan Hall for his comments and suggestions regarding the copolar and $\Delta$-spaces and to the referee for his/her very thorough report. The authors thank Andries Brouwer, Robert Coulter, Hans Cuypers, Gary Ebert and Chris Godsil for useful suggestions regarding the topics contained in this paper. Jack Koolen appreciates the comments of Koen Thas concerning regular points in generalized quadrangles. 

This work was partially supported by a grant from the Simons Foundation ($\#209309$ to Sebastian M. Cioab\u{a}). Kijung Kim was supported by the National Research Foundation of Korea Grant funded by the Korean Government[NRF-2010-355-C00002]. The second author's work was done while Kijung Kim was at POSTECH. Kijung Kim also thanks the Faculty of Mathematics, POSTECH, for warm hospitality. Jack H. Koolen was partially supported by the Basic Science Research Program through the National Research Foundation of Korea (NRF) funded by the Ministry of Education, Science and Technology (GrantNo. 2009-0089826).

\end{document}